\documentclass[12pt]{amsart}
\usepackage{amssymb,amsmath,amsthm,mathrsfs}
\numberwithin{equation}{section}

\newtheorem{thm}{Theorem}[section]
\newtheorem{lemma}[thm]{Lemma}
\newtheorem{remark}[thm]{Remark}

\newtheorem{definition}[thm]{Definition}

\allowdisplaybreaks

\begin{document}

\title[Convergence Rates of ESD] {Convergence Rates of Spectral Distribution of Large Dimensional Quaternion Sample Covariance Matrix }

\author{HUIQIN LI,\ \ ZHIDONG BAI}
\thanks{ H. Q. Li was partially supported by a grant CNSF 11301063; Z. D. Bai was partially supported by CNSF  11171057, the Fundamental Research Funds for the Central Universities, PCSIRT, and the NUS Grant R-155-000-141-112.}

\address{KLASMOE and School of Mathematics \& Statistics, Northeast Normal University, Changchun, P.R.C., 130024.}
\email{lihq118@nenu.edu.cn}
\address{KLASMOE and School of Mathematics \& Statistics, Northeast Normal University, Changchun, P.R.C., 130024.}
\email{baizd@nenu.edu.cn}

\subjclass{Primary 15B52, 60F15, 62E20;
Secondary 60F17}

\maketitle

\begin{abstract}
In this paper, we study the convergence rates of empirical spectral distribution of large dimensional quaternion sample covariance matrix. Assume that the entries of $\mathbf X_n$ ($p\times n$) are independent quaternion random variables with mean zero, variance 1 and uniformly bounded sixth moments. Denote $\mathbf S_n=\frac{1}{n}\mathbf X_n\mathbf X_n^*$. Using Bai inequality, we prove that the expected empirical spectral distribution (ESD) converges to the limiting Mar${\rm \check{c}}$enko-Pastur distribution with the ratio of the dimension to sample size $y_p=p/n$ at a rate of $O\left(n^{-1/2}a_n^{-3/4}\right)$ when $a_n>n^{-2/5}$ or $O\left(n^{-1/5}\right)$ when $a_n\le n^{-2/5}$, where $a_n=(1-\sqrt{y_p})^2$ is the lower bound for the M-P law. Moreover, the rates for both the convergence in probability and the almost sure convergence are also established. The weak convergence rate of the ESD is $O\left(n^{-2/5}a_n^{-2/5}\right)$ when $a_n>n^{-2/5}$ or $O\left(n^{-1/5}\right)$ when $a_n\le n^{-2/5}$. The strong convergence rate of the ESD is $O\left(n^{-2/5+\eta}a_n^{-2/5}\right)$ when $a_n>n^{-2/5}$ or $O\left(n^{-1/5}\right)$ when $a_n\le n^{-2/5}$ for any $\eta>0$.

{\bf Keywords}: Empirical Spectral Distribution;  Mar${\rm \check{c}}$enko-Pastur Law; Weak Convergence Rate; Strong Convergence Rate;  Quaternion Sample Covariance Matrix.

\end{abstract}

\section{Introduction}

Let $A$ be a $p \times p$ Hermitian matrix and denote its eigenvalues  by ${s_j}, j = 1,2, \cdots, p$. The empirical spectral distribution (ESD) of $A$ is defined by
$${F^A}\left(x\right) =\frac{1}{p}\sum\limits_{j = 1}^p {I\left({s _j} \le x\right)},$$ where ${I\left(D\right)}$ is the indicator function of an event ${D}$. Huge data sets with large dimension and large sample size lead to failure of the applications of the classical limit theorems. In recent decades, the theory of random matrices (RMT) has been actively developed which enables us to find the solutions to this issue. The sample covariance matrix is one of the most important  random matrices in RMT, which can be traced back to Wishart (1928) \cite{wishart1928generalised}. In \cite{marchenko1967distribution},  Mar${\rm \breve{c}}$enko and Pastur  proved that ESD of  large  dimensional complex sample covariance matrices tends to the M-P law $F_y\left(x\right)$ with the density function
\begin{align*}
f_y\left(x\right)=  {\begin{cases}
{\frac{1}{{2\pi xy{\sigma ^2}}}\sqrt {\left(b - x\right)\left(x - a\right)} ,}& a \le x \le b,\\
{0,} & otherwise,
\end{cases}}
\end{align*}
where $a = {\sigma ^2}{\left(1 - \sqrt y \right)^2}$, $b = {\sigma ^2}{\left(1 + \sqrt y \right)^2}$, ${\sigma ^2}$  is the scale parameter, and the constant $y$ is the limiting ratio of  dimension $p$ to sample size $n$.
If $y > 1$, $F_y\left(x\right)$ has a point mass $1 - 1/y$ at the origin. After the limiting spectral distribution (LSD) of the sample covariance matrices is found,   two important problems arise. The first is the bound on extreme eigenvalues; the second is the convergence rate of the ESD with respect to sample size. Yin, Bai and Krishnaiah (1988) \cite{Yin1988} proved that the largest eigenvalue of the large dimensional real sample covariance matrix tends to $\sigma^2\left(1+\sqrt y\right),\ a.s.$. Bai and Yin (1993) \cite{BaiYin1993} established the conclusion that the smallest eigenvalue of the large dimensional real sample covariance matrix strongly converges to $\sigma^2\left(1-\sqrt y\right)$. For convergence rate, since Bai \cite{Bai1993a} established a Berry-Essen type inequality, much work has been done (see \cite{Bai1993b,Bai201268,Bai199795,Bai2003,gotze2004rate,gotze2010rate}, among others). Here the readers are referred to three books \cite{anderson2010introduction,bai2010spectral,mehta2004random} for more details.

As the wide applications of quaternions and quaternion matrices  in quantum physics, robot technology and artificial satellite attitude control, etc., it is necessary to study the quaternion sample covariance matrix. In \cite{li2013convergence}, it was proved that the ESD  of  large  dimensional quaternion sample covariance matrix tends to the M-P law. From \cite{li2013extreme}, we have known the limits of extreme eigenvalues of quaternion sample covariance matrix. Convergence rates of the ESD of the quaternion sample covariance matrix are considered in this paper.

In what follows, we introduce some notations about quaternions. The quaternion base can be represented by four $2\times 2$ matrices as
\begin{align*}
\mathbf {e} = \left( \begin{array}{cc}
1&0\\
0&1\\
\end{array} \right),
\mathbf i = \left( \begin{array}{cc}
i&0\\
0&- i\\
\end{array} \right),
\mathbf j = \left( \begin{array}{cc}
0&1\\
-1&0\\
\end{array}\right),
\mathbf k = \left( \begin{array}{cc}
0&i\\
i&0\\
\end{array}\right),\end{align*}
where $i=\sqrt{-1}$ denotes the imaginary unit. Thus, a quaternion can be written by a $2\times 2$ complex matrix as
\begin{align*}
x = a \cdot \mathbf e + b \cdot \mathbf i + c \cdot \mathbf j + d \cdot \mathbf k =\left( {\begin{array}{*{20}{c}}
a+bi &c+di\\
{ - c+di }&{a-bi }
\end{array}} \right)
\triangleq\left( {\begin{array}{*{20}{c}}
\lambda &\omega\\
-\overline{\omega }&\overline{\lambda}
\end{array}} \right)
\end{align*}
where the coefficients  $a,b,c,d$ are real. The conjugate of $x$ is defined as
$$\bar x = a \cdot \mathbf e - b \cdot \mathbf i - c \cdot \mathbf j - d \cdot \mathbf k=\left( {\begin{array}{*{20}{c}}
a-bi &-c-di \\
{ c-di }&{a+bi }
\end{array}} \right)
=\left( {\begin{array}{*{20}{c}}
\overline{\lambda} &-\omega\\
\overline{\omega }&\lambda
\end{array}} \right) $$
and its norm as
 $$\left\| x \right\| = \sqrt {{a^2} + {b^2} + {c^2} + {d^2}}=\sqrt{\left|\lambda\right|^2+\left|\omega\right|^2}.$$
More details can be found in \cite{adler1995quaternionic,finkelstein1962foundations,zhang1995numerical,kuipers1999quaternions,mehta2004random,zhang1997quaternions,zhang1994numerical}. It is worth mentioning that any \ $n\times n$ quaternion matrix \ $\mathbf Y$ can be represented as a \ $2n\times 2n$ complex matrix \ $\psi(\mathbf Y)$. Consequently, we can deal with quaternion matrices as complex matrices.

The following two tools play a key role in establishing the convergence
rates of the  ESD. The first is Bai inequality:
\begin{lemma}{(Bai inequality in \cite{Bai1993a})}\label{bai1993}
Let $F$ be a distribution function and $G$ be a function of bounded variation satisfying $\int |F(x)-G(x)|\mathrm{d}x< \infty.$ Denote their Stieltjes transforms by $f(z)$ and $g(z)$, respectively, where $z=u+iv \in \mathbb{C^+}$. Then we have
\begin{equation}\label{eq:8}
\begin{split}
  \left\|F-G\right\|&\overset{\text{def}}=\sup \limits_{x} \left|F(x)-G(x)\right| \\
  &\leq \frac{1}{\pi\left(1-\kappa\right)\left(2\gamma-1\right)} {\bigg [} \int_{-A}^{A}\left|f\left(z\right)-g\left(z\right)\right|\mathrm{d}u \\
  &+2\pi v^{-1} \int_{\left|x\right|>B}\left|F\left(x\right)-G\left(x\right)\right|\mathrm{d}x \\
  &+v^{-1} \sup \limits_{x} \int_{\left|s\right|\leq 2va}\left|G\left(x+s\right)-G\left(x\right)\right|\mathrm{d}s \bigg],
\end{split}
\end{equation}
where $a$, $\gamma$, $A$ and $B$ are positive constants such that $A>B$, $$\gamma=\frac{1}{\pi}\int_{|u|<a}\frac{1}{u^2+1}\mathrm{d}u>\frac{1}{2}, \ {\rm and} \  \kappa=\frac{4B}{\pi(A-B)(2\gamma-1)}<1.$$
\end{lemma}
The other is the form of the inverse of some matrices related to quaternions:
\begin{lemma}[see \cite{li2013convergence} or \cite{yin2013semicircular}]\label{yin}
For all $n\geq1$, if a complex  matrix  $\mathbf \Omega_n$ is   invertible and of Type-\uppercase\expandafter{\romannumeral3}, then $\mathbf \Omega_n^{-1}$ is a Type-\uppercase\expandafter{\romannumeral1} matrix.
\end{lemma}

\noindent
In Lemma \ref{yin}, the Type-\uppercase\expandafter{\romannumeral3} and Type-\uppercase\expandafter{\romannumeral1} are defined as follows:
\begin{definition} A matrix is called Type-\uppercase\expandafter{\romannumeral1} matrix if it has the following structure:
\[\left( {\begin{array}{*{20}{c}}
{{t_1}}&0&{{a_{12}}}&{{b_{12}}}& \cdots &{{a_{1n}}}&{{b_{1n}}}\\
0&{{t_1}}&{{c_{12}}}&{{d_{12}}}& \cdots &{{c_{1n}}}&{{d_{1n}}}\\
{{d_{12}}}&{ - {b_{12}}}&{{t_2}}&0& \cdots &{{a_{2n}}}&{{b_{2n}}}\\
{ - {c_{12}}}&{{a_{12}}}&0&{{t_2}}& \cdots &{{c_{2n}}}&{{d_{2n}}}\\
 \vdots & \vdots & \vdots & \vdots & \ddots & \vdots & \vdots \\
{{d_{1n}}}&{ - {b_{1n}}}&{{d_{2n}}}&{ - {b_{2n}}}& \cdots &{{t_n}}&0\\
{ - {c_{1n}}}&{{a_{1n}}}&{ - {c_{2n}}}&{{a_{2n}}}& \ldots &0&{{t_n}}
\end{array}} \right).\]
Here all the  entries are  complex.
\end{definition}
\begin{definition}A matrix is called Type-\uppercase\expandafter{\romannumeral3} matrix if it has the following structure:
\[\left( {\begin{array}{*{20}{c}}
{{t_1}}&0&{{a_{12}} }&{{b_{12}} }& \cdots &{{a_{1n}}}&{{b_{1n}} }\\
0&{{t_1}}&{ - {{\bar b}_{12}} }&{{{\bar a}_{12}} }& \cdots &{ - {{\bar b}_{1n}} }&{{{\bar a}_{1n}}}\\
{{{\bar a}_{12}} }&{ - {b_{12}} }&{{t_2}}&0& \cdots &{{a_{2n}} }&{{b_{2n}} }\\
{{{\bar b}_{12}} }&{{a_{12}}}&0&{{t_2}}& \cdots &{ - {{\bar b}_{2n}} }&{{{\bar a}_{2n}} }\\
 \vdots & \vdots & \vdots & \vdots & \ddots & \vdots & \vdots \\
{{{\bar a}_{1n}}}&{ - {b_{1n}} }&{{{\bar a}_{2n}}}&{ - {b_{2n}} }& \cdots &{{t_n}}&0\\
{{{\bar b}_{1n}} }&{{a_{1n}} }&{{{\bar b}_{2n}} }&{{a_{2n}} }& \ldots &0&{{t_n}}
\end{array}} \right).\]
Here all the variables are complex numbers.
\end{definition}

\section{Main theorem}

In this section, we establish the main theorems about convergence rates of the ESD of the quaternion sample covariance matrix. They can be stated as follows.
\begin{thm}\label{th:1}
Suppose that $\mathbf X_n = ({x_{jk}^{\left(n\right)}})_{p\times n}$ is a quaternion random matrix whose entries are independent. Furthermore, assume that $${\rm E}x_{jk}^{\left(n\right)}=0, {\rm E}\left\|x_{jk}^{\left(n\right)}\right\|^2 =1,\sup_n\sup_{jk}{\rm E} \left\|x_{jk}^{\left(n\right)}\right\|^6\le M.$$
Then, denoting
the ESD of $\mathbf S_n=\frac{1}{n}{\mathbf X_n}{\mathbf X_n^*}$ as \ $F^{\mathbf S_n}$,
we have
\begin{align}\label{al:1}
\left\|{\rm E}F^{\mathbf S_n}-F_{y_p}\right\|=\begin{cases}
O\left(n^{-1/2}a_n^{-3/4}\right),&if \ a_n>n^{-2/5},\\
O\left(n^{-1/5}\right),&otherwise,\end{cases}
\end{align}
where $y_p=p/n$ and $a_n=\left(1-\sqrt{y_p}\right)^2$.
\end{thm}
\begin{remark}
For brevity, we shall drop the superscript $(n)$ from the variables and denote $\left\|{\rm E}F^{\mathbf S_n}-F_{y_p}\right\|$ by $\Delta$.
\end{remark}
\begin{remark}
Note that
\begin{align*}
\left\|{\rm E}F^{\mathbf S_n}-F_{y}\right\|\ge \left\|F_y-F_{y_p}\right\|-\left\|{\rm E}F^{\mathbf S_n}-F_{y_p}\right\|.
\end{align*}
Consequently, the convergence rate of $\left\|{\rm E}F^{\mathbf S_n}-F_{y}\right\|$ relies on that of $\left|y_p-y\right|$. Therefore, it is impossible to establish  the convergence rate of $\left\|{\rm E}F^{\mathbf S_n}-F_{y}\right\|$, unless we know the rate of $\left|y_p-y\right|$. Thus, we  have to consider the convergence rate  of $\left\|{\rm E}F^{\mathbf S_n}-F_{y_p}\right\|$.
\end{remark}
\begin{remark}
To prove Theorem \ref{th:1}, it suffices to show that (\ref{al:1}) is true when $y_p\le 1$.

In fact, for $y_p>1$, write $\mathbf W_p=\frac{1}{p}\mathbf{X_n^*X_n}$ and denote by $G_n\left(x\right)$ the ESD of $\mathbf W_p$. It is known that $\mathbf{X_n^*X_n}$ and $\mathbf{X_nX_n^*}$ have the same nonzero eigenvalues. By calculation, one gets
\begin{align*}
F^{\mathbf S_n}\left(x\right)=y_p^{-1}G_n\left(y_p^{-1}x\right)+\left(1-y_p^{-1}\right)I\left(x\ge 0\right)
\end{align*}
which implies that
\begin{align*}
\left\|F^{\mathbf S_n}-F_{y_p}\right\|=y_p^{-1}\left\|G_n-F_{1/y_p}\right\|.
\end{align*}
Therefore, the convergence rate for $y_p>1$ can turn into that for $1/y_p<1$.
\end{remark}
\begin{thm}\label{th:2}
Under the assumptions in Theorem \ref{th:1}, we have
\begin{align*}
\left\|F_p-F_{y_p}\right\|=\begin{cases}
O_p\left(n^{-1/5}\right),& if \ a_n<n^{-2/5},\\
O_p\left(n^{-2/5}a_n^{-2/5}\right),& if \ n^{-2/5}\le a_n<1.\end{cases}
\end{align*}
\end{thm}
\begin{thm}\label{th:3}
Under the assumptions in Theorem \ref{th:1}, we have
\begin{align*}
\left\|F_p-F_{y_p}\right\|=\begin{cases}
O_{a.s.}\left(n^{-1/5}\right),& if \ a_n<n^{-2/5},\\
O_{a.s.}\left(n^{-2/5+\eta}a_n^{-2/5}\right),& if \ n^{-2/5}\le a_n<1.\end{cases}
\end{align*}
\end{thm}

\section{Preliminaries}
Before proving the Theorem \ref{th:1}, we first truncate the entries of the matrix and renormalize them in order to obtain  the bound of $\left\|x_{jk}\right\|$, without changing the convergence rate of $F^{\mathbf S_n}$. The results are listed in Subsection  \ref{se:1}.
\subsection{Truncation}
We truncate the variables $x_{jk}$ at $n^{1/4}$. Denote the truncated entries and matrix by $\widetilde x_{jk}=x_{jk}I\left(\left\|x_{jk}\right\|<n^{1/4}\right)$ and $\mathbf{\widetilde X_n}=\left(\widetilde x_{jk}\right)$, respectively. Furthermore, let $\widetilde F^{\mathbf S_n}$ denote the ESD of the quaternion sample covariance matrix $\frac{1}{n}\mathbf{\widetilde X_n}\mathbf{\widetilde X_n}^*$. Then, by rank inequality Lemma \ref{le:1}, we have
\begin{align}\label{eq:1}
\left\|F^{\mathbf S_n}-\widetilde F^{\mathbf S_n}\right\|\le\frac{1}{p}\sum_{jk}I\left(\left\|x_{jk}\right\|\ge n^{1/4}\right).
\end{align}

Note that
\begin{align*}
{\rm E}\left(n^{1/2}{p}^{-1}\sum_{jk}I\left(\left\|x_{jk}\right\|\ge n^{1/4}\right)\right)
\le& n^{1/2}{p}^{-1}\sum_{jk}{\rm P}\left(\left\|x_{jk}\right\|\ge n^{1/4}\right)\\
\le& n^{-1}{p}^{-1}\sum_{jk}{\rm E}\left\|x_{jk}\right\|^6I\left(\left\|x_{jk}\right\|\ge n^{1/4}\right)\\
\le & M
\end{align*}
and
\begin{align*}
{\rm Var}\left(n^{1/2}{p}^{-1}\sum_{jk}I\left(\left\|x_{jk}\right\|\ge n^{1/4}\right)\right)
\le& n{p}^{-2}\sum_{jk}{\rm P}\left(\left\|x_{jk}\right\|\ge n^{1/4}\right)\\
\le& n^{-1/2}{p}^{-2}\sum_{jk}{\rm E}\left\|x_{jk}\right\|^6I\left(\left\|x_{jk}\right\|\ge n^{1/4}\right)\\
=&M n^{-1/2}.
\end{align*}
By Bernstein's inequality (see Lemma \ref{le:2}), for all small $\varepsilon>0$ and large $p, n$, it follows that
\begin{align*}
&{\rm P}\left(\left|\sum_{jk}n^{1/2}p^{-1}I\left(\left\|x_{jk}\right\|\ge n^{1/4}\right)\right|\ge M+\varepsilon \right)\\
\le &{\rm P}\left(\left|\sum_{jk}n^{1/2}p^{-1}\left[I\left(\left\|x_{jk}\right\|\ge n^{1/4}\right)-{\rm P}\left(\left\|x_{jk}\right\|\ge n^{1/4}\right)\right]\right|\ge \varepsilon \right)\\
\le&2\exp\left\{-\frac{\varepsilon^2}{2(Mn^{-1/2}+n^{1/2}p^{-1}\varepsilon)}\right\}\\
\triangleq&2\exp\left\{-cn^{1/2}\right\}\ (c>0)
\end{align*}
which is summable. Applying Borel-Cantelli lemma, we have
\begin{align}\label{eq:2}
&n^{1/2}p^{-1}\sum_{jk}I\left(\left\|x_{jk}\right\|\ge n^{1/4}\right)\le M+\varepsilon \ {a.s.}
\end{align}
Together with (\ref{eq:1}), (\ref{eq:2}) and Lemma \ref{le:3}, one has
\begin{align*}
L\left(F^{\mathbf S_n},\widetilde F^{\mathbf S_n}\right)=O_{{a.s.}}\left(n^{-1/2}\right).
\end{align*}
\subsection{Centralization}
 Write $\widehat x_{jk}=\widetilde x_{jk}-{\rm E}\widetilde x_{jk}$ and $\mathbf{\widehat X_n}=\left(\widehat x_{jk}\right)$. Denote by $\widehat F^{\mathbf S_n}$ the ESD of the quaternion sample covariance matrix $\frac{1}{n}\mathbf{\widehat X_n}\mathbf{\widehat X_n}^*$. Using Lemma \ref{le:4}, we get
 \begin{align}\label{eq:3}
 L\left(\widetilde F^{\mathbf S_n},\widehat F^{\mathbf S_n}\right)\le2\left\|\frac{1}{\sqrt n}\mathbf{\widetilde X_n}\right\|_2\left\|\frac{1}{\sqrt n}{\rm E}\mathbf{\widetilde X_n}\right\|_2+\left\|\frac{1}{\sqrt n}{\rm E}\mathbf{\widetilde X_n}\right\|_2^2.
 \end{align}
By elementary calculation, one obtains
 \begin{equation}\label{eq:4}
 \begin{split}
 \left\|\frac{1}{\sqrt n}{\rm E}\mathbf{\widetilde X_n}\right\|_2
 \le&\sqrt n\max_{jk}{\rm E}\left\|x_{jk}\right\|I\left(\left\|x_{jk}\right\|\ge n^{1/4}\right)\\
 \le&n^{-3/4}{\rm E}\left\|x_{jk}\right\|^6=O\left(n^{-3/4}\right).
 \end{split}
 \end{equation}
 By Remark 2.3 in \cite{li2013extreme}, we know that
 \begin{align*}
 \limsup\left\|\frac{1}{\sqrt n}\mathbf{\widehat X_n}\right\|_2\le1+\sqrt y,{a.s.}
 \end{align*}
which implies that
  \begin{equation}\label{eq:5}
  \begin{split}
 \limsup\left\|\frac{1}{\sqrt n}\mathbf{\widetilde X_n}\right\|_2\le&\limsup\left\|\frac{1}{\sqrt n}\mathbf{\widehat X_n}\right\|_2+ \left\|\frac{1}{\sqrt n}{\rm E}\mathbf{\widetilde X_n}\right\|_2\\
 \le&1+\sqrt y,{a.s..}
 \end{split}
 \end{equation}
 Together with (\ref{eq:3}), (\ref{eq:4}) and (\ref{eq:5}), we can show that
 \begin{align*}
 L\left(\widetilde F^{\mathbf S_n},\widehat F^{\mathbf S_n}\right)=O_{{a.s.}}\left(n^{-3/4}\right).
 \end{align*}
\subsection{Rescaling}
Write $\check{x}_{jk}=\sigma_{jk}^{-1}\widehat x_{jk}$ and $\mathbf{\check{X}_n}=\left(\check{x}_{jk}\right)$ where $\sigma_{jk}^2={\rm E}\left\|\widehat{x}_{jk}\right\|^2$.  Moreover, let $\check F^{\mathbf S_n}$ denote the ESD of the quaternion sample covariance matrix $\frac{1}{n}\mathbf{\check X_n}\mathbf{\check X_n}^*$. By Lemma \ref{le:4}, one has
 \begin{align}\label{eq:6}
 L\left(\check F^{\mathbf S_n},\widehat F^{\mathbf S_n}\right)\le2\left\|\frac{1}{\sqrt n}\mathbf{\widehat X_n}\right\|_2\left\|\frac{1}{\sqrt n}
 \left(\mathbf{\widehat X_n}-\mathbf{\check X_n}\right)\right\|_2+\left\|\frac{1}{\sqrt n}\left(\mathbf{\widehat X_n}-\mathbf{\check X_n}\right)\right\|_2^2.
 \end{align}
 By element calculation, we get
 \begin{equation}\label{eq:7}
 \begin{split}
\left\|\frac{1}{\sqrt n} \left(\mathbf{\widehat X_n}-\mathbf{\check X_n}\right)\right\|_2^2\le&\frac{2}{n}\sum_{jk}\left\|\widehat x_{jk}\right\|^2\left(\sigma_{jk}^{-1}-1\right)^2\\
\le&2p\max_{jk}\left|\sigma_{jk}^{-1}-1\right|^2\frac{1}{np}\sum_{jk}{\rm E}\left\|\widehat x_{jk}\right\|^2, \ a.s.\\
=&2p\max_{jk}\left|\sigma_{jk}-1\right|^2, \ a.s.\\
=&O_{{a.s.}}\left(n^{-1}\right)
\end{split}
 \end{equation}
 where the second inequality follows from that
 \begin{align*}
{\rm E}\left(\frac{1}{np}\sum_{jk}\left(\left\|\widehat x_{jk}\right\|^2-{\rm E}\left\|\widehat x_{jk}\right\|^2\right)\right)^2
\le\frac{1}{n^2p^2}\sum_{jk}{\rm E}\left\|\widehat x_{jk}\right\|^4\le Cn^{-2}
 \end{align*}
 and the last equality follows from that
 \begin{align*}
 \left|\sigma_{jk}-1\right|\le&1-\sigma_{jk}^2={\rm E}\left\|x_{jk}\right\|^2I\left(\left\|x_{jk}\right\|\ge n^{1/4}\right)
 \le n^{-1}{\rm E}\left\|x_{jk}\right\|^6=O\left(n^{-1}\right).
 \end{align*}
 From (\ref{eq:6}) and (\ref{eq:7}), we can show that
  \begin{align*}
 L\left(\check F^{\mathbf S_n},\widehat F^{\mathbf S_n}\right)=O_{{a.s.}}\left(n^{-1/2}\right).
 \end{align*}
 \subsection{Conclusion}\label{se:1}
 Combining the three subsections above, Lemma \ref{le:5} and Remark \ref{re:1}, we get
 \begin{align*}
 \left\|F^{\mathbf S_n}-F_{y_p}\right\|_2\le C\max\left\{ \left\|\check{ F}^{\mathbf S_n}-F_{y_p}\right\|_2,\frac{1}{\sqrt{na}+\sqrt[4]{n}}\right\}.
 \end{align*}
For brevity, we still use $x_{jk}$ to denote the variables after truncation and renormalization. Thus, to complete the proof of Theorem \ref{th:1}, we can further assume that
\begin{itemize}
\item 1): ${\rm E}x_{jk}=0$, ${\rm E}\left\|x_{jk}\right\|^2=1$,
\item 2): $\left\|x_{jk}\right\|<n^{1/4}$,
\item 3):$\sup_{jk}{\rm E}\left\|x_{jk}\right\|^6\le M$.
\end{itemize}

\section{Proof of Theorem \ref{th:1}}
The Stieltjes transform of  M-P law $F_{y_p}\left(x\right)$ is given by
\begin{align*}
s\left(z\right)=&\int_{-\infty}^{+\infty}\frac{1}{x-z}d{F_{y_p}}\left(x\right)
=\frac{{1 - y_p - z + \sqrt {{{\left(z - 1 - y_p\right)}^2} - 4y_p} }}{{2y_pz}}
\end{align*}
where $z=u+\upsilon i\in\mathbb{C}^+$. And the Stieltjes transform of ${F^{\mathbf S_n}}\left(x\right)$ is
$$s_p\left(z\right)=\int_{-\infty}^{+\infty}\frac{1}{x-z}d{F^{\mathbf S_n}}\left(x\right)=\frac{1}{2p}{\rm tr}\left(\mathbf S_n-z\mathbf I_{2p}\right)^{-1}.$$
Applying Lemma \ref{le:6}, one has\[{s_p}\left(z\right) = \frac{1}{{2p}}\sum\limits_{k = 1}^p {{\rm tr}\left(\frac{1}{n}\boldsymbol{\phi}_k^{\prime}\bar{\boldsymbol{\phi}}_k - z{\mathbf I_2} - \frac{1}{{{n^2}}}\boldsymbol{\phi}_k^{\prime}{\mathbf X}_{nk}^*{{\left(\frac{1}{n}{\mathbf X_{nk}}{\mathbf X_{nk}^*} - z{\mathbf I_{2p - 2}}\right)}^{-1}}{\mathbf X_{n k}}\bar{\boldsymbol{\phi}}_k\right )}^{ - 1}\]
where \ ${\mathbf X_{nk}}$ is the matrix resulting from deleting the $k$-th quaternion row of \ $\mathbf X_n$, and $\boldsymbol{\phi}_k^{\prime}$  is the quaternion vector of order $1\times n$ obtained from the $k$-th quaternion row of \ $\mathbf X_n$. Set \begin{align*}
{\boldsymbol \varepsilon _k} &=\frac{1}{n}\boldsymbol{\phi}_k^{\prime}\bar{\boldsymbol{\phi}}_k - z{\mathbf I_2} - \frac{1}{{{n^2}}}\boldsymbol{\phi}_k^{\prime}{\mathbf X}_{nk}^*{{\left(\frac{1}{n}{\mathbf X_{nk}}{\mathbf X_{nk}^*} - z{\mathbf I_{2p - 2}}\right)}^{-1}}{\mathbf X_{nk}}\bar{\boldsymbol{\phi}}_k \\
&- \left(1 - z - {y_p} - {y_p}z{\rm E}{s_p}\left(z\right)\right){\mathbf I_2}.
\end{align*}
We can show that $${\rm E}{s_p}\left(z\right) = \frac{1}{{1 - z - {y_p} - {y_p}z{\rm E}{s_p}\left(z\right)}} + {\delta _n}$$
where
\begin{align}\label{al11}
{\delta _n} = & - \frac{1}{2p\left({1 - z - {y_p} - {y_p}z{\rm E}{s_p}\left(z\right)}\right)}\notag\\
\times&\sum\limits_{k = 1}^p {\rm Etr}\left\{{\boldsymbol\varepsilon _k}{\left(\left(1 - z - {y_p} - {y_p}z{\rm E}{s_p}(z)\right){\mathbf I_2} + {\boldsymbol\varepsilon _k}\right)^{ - 1}}\right\}.
\end{align}
From \cite{li2013convergence}, we have known that the root of the equation above is
\begin{align*}
{\rm E}s_p\left(z\right)=\frac{1-z-y_p+y_pz\delta_n+\sqrt{\left(1-z-y_p-y_pz\delta_n\right)^2-4y_pz}}{2y_pz}.
\end{align*}

To begin with, we estimate the first integral in (\ref{eq:8}). Since
\begin{align*}
&\left|{\rm E}s_p\left(z\right)-s\left(z\right)\right|\\
\le&\left|\frac{\delta_n}{2}\right|\left[1+\frac{\left|2\left(z+y_p-1\right)+y_pz\delta_n\right|}
{\sqrt{\left(z+y_p-1\right)^2-4y_pz}+\sqrt{\left(z+y_p-1+y_pz\delta_n\right)^2-4y_pz}}\right],
\end{align*}
we need to find a bound for $\left|\delta_n\right|$. For brevity, we shall use the following notation:
\begin{align*}
&v_{y_p}=\sqrt a_n+ \sqrt v=1-\sqrt y_p+\sqrt v\\
&\mathbf S_{nk}=\frac{1}{n}{\mathbf X_{nk}}{\mathbf X_{nk}^*}\\
&b_n=b_n\left(z\right)=\frac{1}{{ z + {y_p}-1+ {y_p}z{\rm E}{s_p}\left(z\right)}}\\
&\xi_k=\xi_k\left(z\right)=\bigg(\big(z + {y_p}-1 +{y_p}z{\rm E}{s_p}(z)\big){\mathbf I_2} - {\boldsymbol\varepsilon _k}\bigg)^{-1}.
\end{align*}
Using Lemma \ref{yin}, we get the form of \ $\left({\mathbf S_{n}} - z{\mathbf I_{2p}}\right)^{-1}$ as
\[\left( {\begin{array}{*{20}{c}}
{{t_1}}&0&{{a_{12}}}&{{b_{12}}}& \cdots \\
0&{{t_1}}&{ {{c}_{12}}}&{{{d}_{12}}}& \cdots \\
{{{d}_{12}}}&{ - {b_{12}}}&{{t_2}}&0& \cdots \\
{{{-c }_{12}}}&{{a_{12}}}&0&{{t_2}}& \cdots \\
 \vdots & \vdots & \vdots & \vdots & \ddots
\end{array}} \right).\]
That is to say, $\boldsymbol\varepsilon_k$ is a scalar matrix. Denote by $\boldsymbol\alpha_k$ the first column of $\boldsymbol{\phi}_k$ and by $\boldsymbol\beta_k$  the second column of $\boldsymbol{\phi}_k$,  
 then, $\boldsymbol\varepsilon_k=\theta_k\mathbf I_2$
where $\theta_k=\frac{1}{n}{\boldsymbol\alpha_k^{\prime}}{{\bar {\boldsymbol\alpha} }_k} - z - \frac{1}{{{n^2}}}{\boldsymbol\alpha_k^{\prime}}{\mathbf X_{nk}^*}{{\left( \mathbf S_{nk}- z{\mathbf I_{2p - 2}}\right)}^{-1}}{\mathbf X_{nk}}{{\bar {\boldsymbol\alpha} }_k}-\left(1 - z - {y_p} - {y_p}z{\rm E}{s_p}\left(z\right)\right).$ Rewrite
$$\xi_k=-\frac{1}{\frac{1}{n}{\boldsymbol\alpha_k^{\prime}}{{\bar {\boldsymbol\alpha} }_k} - z - \frac{1}{{{n^2}}}{\boldsymbol\alpha_k^{\prime}}{\mathbf X_{nk}^*}{{\left( \mathbf S_{nk}- z{\mathbf I_{2p - 2}}\right)}^{-1}}{\mathbf X_{nk}}{{\bar {\boldsymbol\alpha} }_k}}\mathbf I_2.$$
Noting that
\begin{align*}
&\Im\left({\frac{1}{n}{\boldsymbol\alpha_k^{\prime}}{{\bar {\boldsymbol\alpha} }_k} - z - \frac{1}{{{n^2}}}{\boldsymbol\alpha_k^{\prime}}{\mathbf X_{nk}^*}{{\left( \mathbf S_{nk}- z{\mathbf I_{2p - 2}}\right)}^{-1}}{\mathbf X_{nk}}{{\bar {\boldsymbol\alpha} }_k}}\right)\\
=&-v\left({1+ \frac{1}{{{n^2}}}{\boldsymbol\alpha_k^{\prime}}{\mathbf X_{nk}^*}{{\left( \mathbf S_{nk}- z{\mathbf I_{2p - 2}}\right)}^{-1}}{{\left( \mathbf S_{nk}- \bar z{\mathbf I_{2p - 2}}\right)}^{-1}}{\mathbf X_{nk}}{{\bar {\boldsymbol\alpha} }_k}}\right)<-v,
\end{align*}
one gets
\begin{align}\label{eq:9}
\left|\frac{1}{\frac{1}{n}{\boldsymbol\alpha_k^{\prime}}{{\bar {\boldsymbol\alpha} }_k} - z - \frac{1}{{{n^2}}}{\boldsymbol\alpha_k^{\prime}}{\mathbf X_{nk}^*}{{\left( \mathbf S_{nk}- z{\mathbf I_{2p - 2}}\right)}^{-1}}{\mathbf X_{nk}}{{\bar {\boldsymbol\alpha} }_k}}\right|\le v^{-1}.
\end{align}
We are now in a position to estimate  $\left|\delta_n\right|$. By (\ref{eq:9}) and the fact $\xi_k=b_n\mathbf I_2+b_n\xi_k\boldsymbol{\varepsilon_k}$, one has
\begin{align*}
\left|\delta_n\right|\le&\frac{1}{2p}\sum_{k=1}^p\left(\left|b_n^2\right|\left|{\rm Etr} \boldsymbol{\varepsilon_k}\right|+\left|b_n^3\right|{\rm E}\left|{\rm tr}\boldsymbol\varepsilon_k^2\right|+\left|b_n^4\right|{\rm E}\left|{\rm tr}\boldsymbol\varepsilon_k^3\right|+\left|b_n^4\right|v^{-1}{\rm E}\left|{\rm tr}\boldsymbol\varepsilon_k^4\right|\right)\\
=&\frac{1}{2p}\sum_{k=1}^p\left(\left|b_n^2\right|\left|{\rm Etr} \boldsymbol{\varepsilon_k}\right|+\frac{\left|b_n^3\right|}{2}{\rm E}\left|{\rm tr}\boldsymbol{\varepsilon_k}\right|^2+\frac{\left|b_n^4\right|}{4}{\rm E}\left|{\rm tr}\boldsymbol{\varepsilon_k}\right|^3+\frac{\left|b_n^4\right|}{8}v^{-1}{\rm E}\left|{\rm tr}\boldsymbol{\varepsilon_k}\right|^4\right).
\end{align*}
Next, we shall complete the estimation of $\left|\delta_n\right|$ by the following four steps under the conditions $v>n^{-1/2}$ and $|b_n|\le 2/\sqrt{y_p|z|}$.
\begin{description}
  \item[Step 1] {\bf the estimator of $\left|{\rm Etr} \boldsymbol{\varepsilon_k}\right|$}.

 By Lemma 3.6 in \cite{li2013convergence}, we have
\begin{align*}
\left|{\rm Etr}\boldsymbol{\varepsilon_k}\right|\le\frac{C}{nv}.
\end{align*}
  \item[Step 2] {\bf the estimator of ${\rm E}\left|{\rm tr}\boldsymbol{\varepsilon_k}\right|^2$.}

Let $\widetilde{\rm E}\left(\cdot\right)$ denote the conditional expectation given $\big\{\mathbf x_j,j=1,\cdots,n;j\neq k\big\}$, then we get
\begin{align}\label{al1}
&{\rm E}\left|{\rm tr}\boldsymbol\varepsilon_k\right|^2\notag\\
\leq&{3}\left[{\rm E}\left|{\rm tr}\boldsymbol\varepsilon_k-\widetilde{\rm E}{\rm tr}\boldsymbol\varepsilon_k\right|^2+{\rm E}\left|\widetilde{\rm E}{\rm tr}\boldsymbol\varepsilon_k-{\rm E}{\rm tr}\boldsymbol\varepsilon_k\right|^2+\left|{\rm E}{\rm tr}\boldsymbol\varepsilon_k\right|^2\right].
\end{align}
By Lemma \ref{B.26} and Lemma \ref{le51}, it follows that
\begin{align}\label{al2}
&{\rm E}\left|{\rm tr}\boldsymbol\varepsilon_k-\widetilde{\rm E}{\rm tr}\boldsymbol\varepsilon_k\right|^2\notag\\
\le&2{\rm E}\left|{\rm tr}\left(\frac{1}{n}\boldsymbol\phi_k^{\prime}\boldsymbol{\bar{\phi}}_k-\mathbf I_2\right)\right|^2
+\frac{2}{n^4}{\rm E}\bigg|\bigg({\rm tr}\boldsymbol{\phi}_k^{\prime}{\mathbf X}_{nk}^*{{\left({\mathbf S_{nk}} - z{\mathbf I_{2p - 2}}\right)}^{-1}}{\mathbf X_{nk}}\bar{\boldsymbol{\phi}}_k\notag\\
&-{\rm tr}{\mathbf X}_{nk}^*{{\left({\mathbf S_{nk}} -z{\mathbf I_{2p - 2}}\right)}^{-1}}{\mathbf X_{nk}}\bigg)\bigg|^2\notag\\
\le&\frac{8}{n}+\frac{C}{n^4}{\rm Etr}\left({\mathbf X}_{nk}^*{{\left({\mathbf S_{nk}} -z{\mathbf I_{2p - 2}}\right)}^{-1}}{\mathbf X_{nk}}{\mathbf X}_{nk}^*{{\left({\mathbf S_{nk}} -\bar{z}{\mathbf I_{2p - 2}}\right)}^{-1}}{\mathbf X_{nk}}\right)\notag\\
\le&C\left[\frac{1}{n}+\frac{\left|u\right|^2}{n^2}{\rm Etr}\left(\left(\mathbf S_{nk}-u\mathbf I_{2p-2}\right)^2+v^2\mathbf I_{2p-2}\right)^{-1}\right]\notag\\
\le&C\left[\frac{1}{n}+\frac{\left|u\right|^2}{nv^2}\left(\Delta+{v}/{v_{y_p}}\right)\right].
\end{align}
Applying Lemma \ref{le:8}, we obtain
\begin{align}\label{al3}
&{\rm E}\left|\widetilde{\rm E}{\rm tr}\boldsymbol{\varepsilon}_k-{\rm E}{\rm tr}\boldsymbol\varepsilon_k\right|^2\notag\\
=&\frac{\left|z\right|^2}{n^2}{\rm E}\left|{\rm tr}\left({\mathbf S_{nk}} - z{\mathbf I_{2p - 2}}\right)^{-1}-{\rm E}{\rm tr}\left({\mathbf S_{nk}} -z{\mathbf I_{2p - 2}}\right)^{-1}\right|^2\notag\\
\le&\frac{C\left|z\right|^2}{n^2}\left[{\rm E}\left|{\rm tr}\left({\mathbf S_{n}} - z{\mathbf I_{2p}}\right)^{-1}-{\rm E}{\rm tr}\left({\mathbf S_{n}} -z{\mathbf I_{2p}}\right)^{-1}\right|^2+\frac{1}{v^2}\right]\notag\\
=&{C\left|z\right|^2}\left[{\rm E}\left|s_p\left(z\right)-{\rm E}s_p\left(z\right)\right|^2+\frac{1}{{n^2}v^2}\right]\notag\\
\le&{C\left|z\right|^2}\left[n^{-2}v^{-4}\left(\Delta+v/v_{y_p}\right)+\frac{1}{{n^2}v^2}\right]\notag\\
\le&\frac{C\left|z\right|^2}{n^2v^4}\left(\Delta+v/v_{y_p}\right).
\end{align}
Together with Step 1, (\ref{al1}), (\ref{al2}), and (\ref{al3}), we get
\begin{align*}
{\rm E}\left|{\rm tr}\boldsymbol{\varepsilon_k}\right|^2
\le&C\left(\frac{1}{n}+\frac{\left|z\right|^2}{nv^2}\left(\Delta+{v}/{v_{y_p}}\right)\right).
\end{align*}
  \item[Step 3] {\bf the estimator of ${\rm E}\left|{\rm tr}\boldsymbol{\varepsilon_k}\right|^4$}.

Similar to (\ref{al1}), the estimand can be written as
\begin{align}\label{al6}
&{\rm E}\left|{\rm tr}\boldsymbol{\varepsilon_k}\right|^4\leq{27}\left[{\rm E}\left|{\rm tr}\boldsymbol\varepsilon_k-\widetilde{\rm E}{\rm tr}\boldsymbol\varepsilon_k\right|^4+{\rm E}\left|\widetilde{\rm E}{\rm tr}\boldsymbol\varepsilon_k-{\rm E}{\rm tr}\boldsymbol\varepsilon_k\right|^4+\left|{\rm E}{\rm tr}\boldsymbol\varepsilon_k\right|^4\right].
\end{align}
First of all, we estimate the first term of righthand side of (\ref{al6}). Using Lemma \ref{B.26} and Lemma \ref{le51}, it follows that
\begin{align}
&{\rm E}\left|{\rm tr}\boldsymbol\varepsilon_k-\widetilde{\rm E}{\rm tr}\boldsymbol\varepsilon_k\right|^4\notag\\
\le&8{\rm E}\left|{\rm tr}\left(\frac{1}{n}\boldsymbol\phi_k^{\prime}\boldsymbol{\bar{\phi}}_k-\mathbf I_2\right)\right|^4
+\frac{8}{n^8}{\rm E}\bigg|{\rm tr}\boldsymbol{\phi}_k^{\prime}{\mathbf X}_{nk}^*{{\left({\mathbf S_{nk}} - z{\mathbf I_{2p - 2}}\right)}^{-1}}{\mathbf X_{nk}}\bar{\boldsymbol{\phi}}_k\notag\\
&-{\rm tr}{\mathbf X}_{nk}^*{{\left({\mathbf S_{nk}} -z{\mathbf I_{2p - 2}}\right)}^{-1}}{\mathbf X_{nk}}\bigg|^4\notag\\
\le&\frac{128}{n^2}+\frac{C}{n^4}{\rm E}\bigg\{\varphi_{8}{\rm tr}\left(\left(\mathbf S_{nk}-z\mathbf I_{2p-2}\right)^{-1}\mathbf S_{nk}\left(\mathbf S_{nk}-\overline z\mathbf I_{2p-2}\right)^{-1}\mathbf S_{nk}\right)^2\notag\\
&+\left[\varphi_4{\rm tr}\left(\left(\mathbf S_{nk}-z\mathbf I_{2p-2}\right)^{-1}\mathbf S_{nk}\left(\mathbf S_{nk}-\overline z\mathbf I_{2p-2}\right)^{-1}\mathbf S_{nk}\right)\right]^2\bigg\}\notag\\
\le&\frac{C}{n^2}+\frac{C}{n^4}\bigg\{\varphi_{8}\left(n+\frac{\left|z\right|^4}{v^2}{\rm Etr}\left(\left(\mathbf S_{nk}-u\mathbf I_{2p-2}\right)^{2}+v^2\mathbf I_{2p-2}\right)^{-1}\right)\notag\\
&+\left[n+{\left|z\right|^2}{\rm Etr}\left(\left(\mathbf S_{nk}-u\mathbf I_{2p-2}\right)^{2}+v^2\mathbf I_{2p-2}\right)^{-1}\right]^2\bigg\}\notag\\
\le&C\left[\frac{1}{n^2}+\frac{\left|z\right|^2}{n^2v^4}\left(\Delta+{v}/v_{y_p}\right)^2\right]\notag
\end{align}
where $\varphi_{8}\le Mn^{1/2}$.
Employing Lemma \ref{le:8}, one has
\begin{align*}
&{\rm E}\left|\widetilde{\rm E}{\rm tr}\boldsymbol{\varepsilon}_k-{\rm E}{\rm tr}\boldsymbol\varepsilon_k\right|^4\\
=&\frac{\left|z\right|^4}{n^4}{\rm E}\left|{\rm tr}\left({\mathbf S_{nk}} - z{\mathbf I_{2p - 2}}\right)^{-1}-{\rm E}{\rm tr}\left({\mathbf S_{nk}} -z{\mathbf I_{2p - 2}}\right)^{-1}\right|^4\\
\le&\frac{C\left|z\right|^4}{n^4}\left[{\rm E}\left|{\rm tr}\left({\mathbf S_{n}} - z{\mathbf I_{2p}}\right)^{-1}-{\rm E}{\rm tr}\left({\mathbf S_{n}} -z{\mathbf I_{2p}}\right)^{-1}\right|^4+\frac{1}{v^4}\right]\\
=&{C\left|z\right|^4}\left[{\rm E}\left|s_p\left(z\right)-{\rm E}s_p\left(z\right)\right|^4+\frac{1}{{n^4}v^4}\right]\\
\le&{C\left|z\right|^4}\left[n^{-4}v^{-8}\left(\Delta+v/v_{y_p}\right)^2+\frac{1}{{n^4}v^4}\right]\\
\le&\frac{C\left|z\right|^4}{n^2v^4}\left(\Delta+v/v_{y_p}\right)^2.
\end{align*}
Combining the two inequalities above with Step 1, (\ref{al6}) can be estimated by
\begin{align*}
{\rm E}\left|{\rm tr}\boldsymbol{\varepsilon_k}\right|^4
\le&C\left[\frac{1}{n^2}+\frac{\left|z\right|^4}{n^2v^4}\left(\Delta+{v}/{v_{y_p}}\right)^2\right].
\end{align*}
  \item[Step 4] {\bf the estimator of ${\rm E}\left|{\rm tr}\boldsymbol{\varepsilon_k}\right|^3$}.

By Cauchy's inequality, Step 2 and Step 3, we can easily acquire
\begin{align*}
{\rm E}\left|{\rm tr}\boldsymbol{\varepsilon_k}\right|^3\le&\left({\rm E}\left|{\rm tr}\boldsymbol{\varepsilon_k}\right|^2\right)^{1/2}\left({\rm E}\left|{\rm tr}\boldsymbol{\varepsilon_k}\right|^4\right)^{1/2}\\
\le&\left[\frac{1}{n^{3/2}}+\frac{\left|z\right|^3}{n^{3/2}v^3}\left(\Delta+{v}/{v_{y_p}}\right)^{3/2}\right].
\end{align*}
\end{description}
Assume that $\left|b_n\right|\le 2/\sqrt{y_p\left|z\right|}$. Then, the four steps above yield
\begin{align}\label{al7}
\left|\delta_n\right|\le C_0n^{-1}v^{-3}\left(\Delta+{v}/{v_{y_p}}\right)^2.
\end{align}

By Lemma \ref{8.21}, if $\left|\delta_n\right|<v/\left[v_{y_p}10\left(A+1\right)^2\right]$, then $\Delta\le Cv/v_{y_p}$. Therefore, our next goal is to find a possible value of the  set $\left\{v:\left|\delta_n\right|\le v/\left[v_{y_p}10\left(A+1\right)^2\right]\right\}$ with $vv_{y_p}\sim Cn^{-1/2}$. Define $$\mathscr{F}=\left\{vv_{y_p}^{1/2}:vv_{y_p}^{1/2}>M_0n^{-1/2},\left|\delta_n\right|\le v/\left[v_{y_p}10\left(A+1\right)^2\right]\right\},$$
where $M_0=\sqrt{\frac{C_0\left(C_2+2\right)^2}{10\left(A+1\right)^2}}$ and $C_2$ is the constant given in Lemma \ref{8.21}. It is not difficult to verify $\mathscr{F}\neq\varnothing$. In fact, by (\ref{al11}), (\ref{eq:9}) and
\begin{align*}
\left|b_n\right|\le\frac{1}{\Im\left({ z + {y_p}-1+ {y_p}z{\rm E}{s_p}\left(z\right)}\right)}\le v^{-1},
\end{align*}
we have
\begin{align*}
\left|\delta_n\right|\le\frac{1}{2pv^2}\sum_{k=1}^p\left[\left|{\rm Etr}\boldsymbol\varepsilon_k\right|+v^{-1}{\rm E}\left|{\rm tr}\boldsymbol\varepsilon_k^2\right|\right]\le\frac{C_1}{nv^5}.
\end{align*}
Choosing $v_0=\sqrt[6]{10C_1\left(1+A\right)^2/n}$, the inequality above turns out to be
\begin{align*}
\left|\delta_n\right|\le v_0/\left[10\left(A+1\right)^2\right].
\end{align*}
This indicates $v_0\in\mathscr{F}$, for all large $n$.

We assert that the infimum of $\mathscr{F}$ is $M_0n^{-1/2}$, which is denoted by $v_1v_{y_p}^{1/2}\left(v_1\right)$. If it is not the case, then by the continuity of various functions involved, there must exist $z_2=u_2+v_2i$ with $u_2\in[-A,A]$, $v_2v_{y_p}^{1/2}\left(v_2\right)\in\mathscr{F}$  and such that $\left|\delta_n\left(z_2\right)\right|=v_2/\left[v_{y_p}\left(v_2\right)10\left(A+1\right)^2\right]$ and $\left|\delta_n\left(z_1\right)\right|\le v_2/\left[v_{y_p}\left(v_2\right)10\left(A+1\right)^2\right]$ for any $z_1=u+iv_2$, $u\in[-A,A]$.
 Then, by Lemma \ref{leb}, $|b_n (z_1)|\le \frac2{\sqrt{y_p|z_1|}}$, the inequality  (\ref{al7}) holds. By Lemma \ref{8.21}, we get $$\Delta\le C_2v_2/v_{y_p}\left(v_2\right).$$
Combining the equality above with $v_2v_{y_p}^{1/2}\left(v_2\right)>M_0n^{-1/2}$, it follows that
\begin{align*}
\left|\delta_n\left(z_2\right)\right|< v_2/\left[v_{y_p}\left(v_2\right)10\left(A+1\right)^2\right].
\end{align*}
This leads to a contradiction with $\left|\delta_n\left(z_2\right)\right|=v_2/\left[v_{y_p}\left(v_2\right)10\left(A+1\right)^2\right]$. Hence, $v_1v_{y_p}^{1/2}\left(v_1\right)=M_0n^{-1/2}$ and $\left|\delta_n\right|\le v_1/\left[v_{y_p}\left(v_1\right)10\left(A+1\right)^2\right]$. By Lemma \ref{leb} and Lemma \ref{8.21}, we get $\Delta\le Cv_1/v_{y_p}\left(v_1\right)$.

We shall complete the proof of Theorem \ref{th:1} in the following two cases.
\begin{description}
  \item[Case 1] When ${a_n}>v_1$, we get
\begin{align*}
v_1\sqrt[4] {a_n}\le v_1v_{y_p}^{1/2}\left(v_1\right)=M_0n^{-1/2} \ \Rightarrow v_1\le M_0 n^{-1/2}/\sqrt [4]{a_n}
\end{align*}
and
\begin{align*}
\Delta\le Cv_1/\sqrt {a_n}\le Cn^{-1/2}{a_n^{-3/4}}.
\end{align*}
  \item[Case 2] When ${a_n}\le v_1$, one acquires
\begin{align*}
v_1^{5/4}\le v_1v_{y_p}^{1/2}\left(v_1\right)=M_0n^{-1/2} \ \Rightarrow v_1\le M_0 n^{-2/5}
\end{align*}
and
\begin{align*}
\Delta\le C\sqrt{v_1}\le Cn^{-1/5}.
\end{align*}
\end{description}
Note that $a_n=v_1$ is equivalent to $v_1=\left(\frac {M_0}{2}\right)^{4/5}n^{-2/5}$. Thus, the above two cases is the same as stated in Theorem \ref{th:1}.
This completes the proof of Theorem \ref{th:1}.

\section{Proof of Theorem \ref{th:2}}
Applying Lemma \ref{bai1993}, Lemma \ref{ex}, and Lemma \ref{ex1}, one has
\begin{align}\label{le1}
&{\rm E}\left\|F_p-F_{y_p}\right\|\notag\\
\le &\bigg(\int_{-A}^A{\rm E}\left(\left|s_p\left(z\right)-s_{y_p}\left(z\right)\right|\right)\mathrm du
+2\pi v^{-1} \int_{\left|x\right|>B}\left|1-{\rm E}F_p\left(x\right)\right|\mathrm{d}x \notag\\
&+v^{-1}\sup_x\int_{\left|t\right|<2av}\left|F_{y_p}\left(x+t\right)-F_{y_p}\left(x\right)\right|\mathrm dt\bigg)\notag\\
\le &C\bigg(\int_{-A}^A{\rm E}\left(\left|s_p\left(z\right)-{\rm E}s_p\left(z\right)\right|\right)\mathrm du+\int_{-A}^A
\left(\left|{\rm E}s_p\left(z\right)-s_{y_p}\left(z\right)\right|\right)\mathrm du\notag\\
&+o\left(n^{-2}\right)+v/v_{y_p}\bigg).
\end{align}
By Lemma \ref{le:8}, we have
\begin{align*}
{\rm E}\left(\left|s_p\left(z\right)-{\rm E}s_p\left(z\right)\right|\right)\le&\left[{\rm E}\left(\left|s_p\left(z\right)-{\rm E}s_p\left(z\right)\right|^2\right)\right]^{1/2}\\
\le&Cn^{-1}v^{-2}\left(\Delta+v/v_{y_p}\right)^{1/2}.
\end{align*}
We will estimate ${\rm E}\left\|F_p-F_{y_p}\right\|$ according to the following two cases.
\begin{description}
  \item[Case 1] When $a_n<n^{-2/5}$, choosing $v=M_1n^{-2/5}$ and due to $\Delta=O\left(n^{-1/5}\right)$, it follows that
\begin{align*}
{\rm E}\left(\left|s_p\left(z\right)-{\rm E}s_p\left(z\right)\right|\right)\le Cn^{-3/10}.
\end{align*}
From the proof of Theorem \ref{th:1}, we have known that
$$\int_{-A}^A
\left(\left|{\rm E}s_p\left(z\right)-s_{y_p}\left(z\right)\right|\right)\mathrm du=O\left(n^{-1/5}\right).$$
Consequently, we obtain
$${\rm E}\left\|F_p-F_{y_p}\right\|=O\left(n^{-1/5}\right).$$
  \item[Case 2] When $a_n>n^{-2/5}$, selecting $v=M_2n^{-2/5}a_n^{1/10}$ and owing to $\Delta=\left(n^{-1/2}a_n^{-3/4}\right)$, one gets
$${\rm E}\left(\left|s_p\left(z\right)-{\rm E}s_p\left(z\right)\right|\right)\le Cn^{-2/5}a_n^{-2/5}.$$
From the proof of Theorem \ref{th:1}, it has been proved that
$$\int_{-A}^A
\left(\left|{\rm E}s_p\left(z\right)-s_{y_p}\left(z\right)\right|\right)\mathrm du=O\left(n^{-1/2}a_n^{-3/4}\right).$$
Hence, we have
$${\rm E}\left\|F_p-F_{y_p}\right\|=O\left(n^{-2/5}a_n^{-2/5}\right).$$
\end{description}
The proof of Theorem \ref{th:2} is complete.

\section{Proof of Theorem \ref{th:3}}
By (\ref{le1}) and the proof of Theorem \ref{th:2}, it suffices to show that
\begin{align*}
\int_{-A}^A{\rm E}\left(\left|s_p\left(z\right)-s_{y_p}\left(z\right)\right|\right)\mathrm du=\begin{cases}
O_{a.s.}\left(n^{-1/5}\right),& if \ a_n<n^{-2/5},\\
O_{a.s.}\left(n^{-2/5+\eta}a_n^{-2/5}\right),& if \ n^{-2/5}\le a_n<1.\end{cases}
\end{align*}
By Lemma \ref{le:8}, we have
\begin{align*}
{\rm E}\left(\left|s_p\left(z\right)-{\rm E}s_p\left(z\right)\right|^{2l}\right)
\le&Cn^{-2l}v^{-4l}\left(\Delta+v/v_{y_p}\right)^{l}.
\end{align*}
We proceed to complete the proof by two cases.
\begin{description}
  \item[Case 1] When $a_n<n^{-2/5}$, choosing $v=M_1n^{-2/5}$, we obtain
\begin{align*}
&n^{2l/5}{\rm E}\left(\left|s_p\left(z\right)-{\rm E}s_p\left(z\right)\right|^{2l}\right)\\
\le&Cn^{2l/5}n^{-2l}v^{-4l}\left(\Delta+v/v_{y_p}\right)^{l}\\
\le& Cn^{-l/5}.
\end{align*}
Then the result follows by choosing $l>5$.
  \item[Case 2] When $a_n>n^{-2/5}$, selecting $v=M_2n^{-2/5}a_n^{1/10}$, one has
\begin{align*}
&n^{2l\left(2/5-\eta\right)}a_n^{4l/5}{\rm E}\left(\left|s_p\left(z\right)-{\rm E}s_p\left(z\right)\right|^{2l}\right)\\
\le&Cn^{2l\left(2/5-\eta\right)}a_n^{4l/5}n^{-2l}v^{-4l}\left(\Delta+v/v_{y_p}\right)^{l}\\
\le&  Cn^{-2l\eta}.
\end{align*}
Then the result follows by choosing $l>1/(2\eta)$.
\end{description}
This completes the proof of Theorem \ref{th:3}.

\section{Some Auxiliary Lemmas}
In this section, we establish three lemmas which are of importance in proving the main theorems.
\begin{lemma}\label{le51}
For $z=u+iv$ with $v>0$, one gets
\begin{align*}
{\rm Etr}\left(\left(\mathbf S_{nk}-u\mathbf I_{2p-2}\right)^2+v^2\mathbf I_{2p-2}\right)^{-1}\le&\frac{Cn}{v^2}\left(\Delta+v/v_{y_p}\right).
\end{align*}
\end{lemma}
\begin{proof}
By Lemma \ref{le:8.17} and Lemma \ref{B.22}, it follows that
\begin{align*}
&{\rm Etr}\left(\left(\mathbf S_{nk}-u\mathbf I_{2p-2}\right)^2+v^2\mathbf I_{2p-2}\right)^{-1}\\
=&\frac{1}{v}{\rm E}\Im{\rm tr}\left(\mathbf S_{nk}-z\mathbf I_{2p-2}\right)^{-1}\\
\le&\frac{1}{v}{\rm E}\Im{\rm tr}\left(\mathbf S_{n}-z\mathbf I_{2p}\right)^{-1}+\frac{1}{v^2}\\
\le&\frac{Cn}{v}\left(\left|{\rm E}s_p\left(z\right)-s_{y_p}\left(z\right)\right|+\left|s_{y_p}\left(z\right)\right|\right)+\frac{1}{v^2}\\
\le&\frac{Cn}{v}\left(\Delta/v+1/\left(\sqrt{y_p}v_{y_p}\right)\right)+\frac{1}{v^2}\\
\le&\frac{Cn}{v^2}\left(\Delta+v/v_{y_p}\right).
\end{align*}
\end{proof}

\begin{lemma}\label{le:8}
If $\left|z\right|<A$, $v>n^{-1/2}$, $\left|b_n\right|\le2/\sqrt{y_p\left|z\right|}$, and $l\ge1$, then
\begin{align*}
{\rm E}\left|s_p\left(z\right)-{\rm E}s_p\left(z\right)\right|^{2l}\le\frac{C}{n^{2l}v^{4l}y_p^{2l}}\left(\Delta+v/v_{y_p}\right)^l
\end{align*}
where $A$ is defined in Lemma \ref{bai1993}.
\end{lemma}
\begin{proof}
Write ${\rm E}_k\left(.\right)$ as the conditional expectation given $\left\{x_{lj};l\le k,j\le n\right\}$. Then
\begin{align*}
s_p\left(z\right)-{\rm E}s_p\left(z\right)=&\frac{1}{p}\sum_{k=1}^p\left[{\rm E}_k{\rm tr}\left(\mathbf S_n-z\mathbf I_p\right)^{-1}-{\rm E}_{k-1}{\rm tr}\left(\mathbf S_n-z\mathbf I_p\right)^{-1}\right]\\
=&\frac{1}{p}\sum_{k=1}^p\gamma_k,
\end{align*}
where
\begin{align*}
\gamma_k=&\left({\rm E}_k-{\rm E}_{k-1}\right)\left[{\rm tr}\left(\mathbf S_n-z\mathbf I_{2p}\right)^{-1}-{\rm tr}\left(\mathbf S_{nk}-z\mathbf I_{2p-2}\right)^{-1}\right]\\
=&-\left({\rm E}_k-{\rm E}_{k-1}\right){\rm tr}\left[\left(\frac{1}{n^2}\boldsymbol{\phi}_k^{\prime}{\mathbf X}_{nk}^*{{\left({\mathbf S_{nk}} - z{\mathbf I_{2p - 2}}\right)}^{-2}}{\mathbf X_{nk}}\bar{\boldsymbol{\phi}}_k+\mathbf I_2\right)\xi_k\right]\\
\triangleq&\left({\rm E}_k-{\rm E}_{k-1}\right)\sigma_k.
\end{align*}
By Lemma \ref{lemma:7}, we have $\left|\sigma_k\right|\le2/v$. Applying Lemma \ref{lemma:5}, one has
\begin{align}\label{al:2}
&{\rm E}\left|s_p\left(z\right)-{\rm E}s_p\left(z\right)\right|^{2l}\le C{n^{-2l}}\left\{{\rm E}\left(\sum_{k=1}^p{\rm E}_{k-1}\left|\gamma_k\right|^2\right)^l+\sum_{k=1}^p{\rm E}\left|\gamma_k\right|^{2l}\right\}.
\end{align}
Note that $\boldsymbol{\varepsilon_k}=\theta_k\mathbf I_2$ and $\xi_k=b_n\mathbf I_2+b_n\xi_k\boldsymbol{\varepsilon_k}$, then, $\gamma_k$ can be written as
\begin{align}\label{al:4}
\gamma_k=&-b_nn^{-2}\bigg({\rm E}_{k}{\rm tr}\boldsymbol{\phi}_k^{\prime}{\mathbf X}_{nk}^*{{\left({\mathbf S_{nk}} - z{\mathbf I_{2p - 2}}\right)}^{-2}}{\mathbf X_{nk}}\bar{\boldsymbol{\phi}}_k\notag\\
&-{\rm E}_{k-1}{\rm tr}{\mathbf X}_{nk}^*{{\left({\mathbf S_{nk}} - z{\mathbf I_{2p - 2}}\right)}^{-2}}{\mathbf X_{nk}}\bigg)
+\left({\rm E}_{k}-{\rm E}_{k-1}\right)b_n\theta_k\sigma_k.
\end{align}
Using Lemma \ref{B.26} and the condition $\left|b_n\right|\le2/\sqrt{y_p\left|z\right|}$, we have
\begin{align}\label{al:5}
&n^{-4}\left|b_n\right|^2{\rm E}_{k-1}\bigg|\bigg[{\rm tr}\boldsymbol{\phi}_k^{\prime}\mathbf X_{nk}^*\left(\mathbf S_{nk}-z\mathbf I_{2p-2}\right)^{-2}\mathbf X_{nk}\bar{\boldsymbol{\phi}}_k\notag\\
& \ \ \ \ \ \ \ \ \ \ \ \ -{\rm tr}\left(\mathbf X_{nk}^*\left(\mathbf S_{nk}-z\mathbf I_{2p-2}\right)^{-2}\mathbf X_{nk}\right)\bigg]\bigg|^2\notag\\
\le&\frac{C}{n^2}\left|b_n\right|^2{\rm E}_{k-1}{\rm tr}\left(\left(\mathbf S_{nk}-z\mathbf I_{2p-2}\right)^{-2}\mathbf S_{nk}\left(\mathbf S_{nk}-\overline z\mathbf I_{2p-2}\right)^{-2}\mathbf S_{nk}\right)\notag\\
\le&\frac{C}{n^2v^2y_p\left|z\right|}{\rm E}_{k-1}{\rm tr}\left(\left(\mathbf S_{nk}-z\mathbf I_{2p-2}\right)^{-1}\mathbf S_{nk}\left(\mathbf S_{nk}-\overline z\mathbf I_{2p-2}\right)^{-1}\mathbf S_{nk}\right)\notag\\
\le&\frac{C}{n^2v^2y_p\left|z\right|}{\rm E}_{k-1}\left[n+\left|z\right|^2{\rm tr}\left(\left(\mathbf S_{nk}-z\mathbf I_{2p-2}\right)^{-1}\left(\mathbf S_{nk}-\overline z\mathbf I_{2p-2}\right)^{-1}\right)\right]\notag\\
\le&\frac{C}{n^2v^2y_p\left|z\right|}{\rm E}_{k-1}\left[n+\frac{\left|z\right|^2}{v}\Im{\rm tr}\left(\mathbf S_{nk}-z\mathbf I_{2p-2}\right)^{-1}\right]\notag\\
\le&\frac{C}{n^2v^2y_p\left|z\right|}{\rm E}_{k-1}\left[n+\frac{\left|z\right|^2}{v^2}+\frac{n\left|z\right|^2}{v}\Im s_p\left(z\right)\right]\notag\\
\le&\frac{C}{nv^3}{\rm E}_{k-1}\left(1+\Im s_p\left(z\right)\right).
\end{align}
Recall that
\begin{align*}
{\rm tr}\varepsilon_k=&{\rm tr}\left(\frac{1}{n}\boldsymbol{\phi}_k^{\prime}\boldsymbol{\bar{\phi}}_k-\mathbf I_2+y_p\mathbf I_2+y_pz{\rm E}s_p\left(z\right)\mathbf I_2-\frac{1}{n^2}\boldsymbol{\phi}_k^{\prime}\mathbf X_{nk}^*\left(\mathbf S_{nk}-z\mathbf I_{2p-2}\right)^{-1}\mathbf X_{nk}\bar{\boldsymbol{\phi}}_k\right)\\
=&{\rm tr}\left(\frac{1}{n}\boldsymbol{\phi}_k^{\prime}\boldsymbol{\bar{\phi}}_k-\mathbf I_2\right)+\frac{2}{n}+\frac{z}{n}\left({\rm Etr}\left(\mathbf S_{n}-z\mathbf I_{2p}\right)^{-1}-{\rm tr}\left(\mathbf S_{nk}-z\mathbf I_{2p-2}\right)^{-1}\right)\\
&-\frac{1}{n^2}\left({\rm tr}\boldsymbol{\phi}_k^{\prime}\mathbf X_{nk}^*\left(\mathbf S_{nk}-z\mathbf I_{2p-2}\right)^{-1}\mathbf X_{nk}\bar{\boldsymbol{\phi}}_k-{\rm tr}\mathbf X_{nk}^*\left(\mathbf S_{nk}-z\mathbf I_{2p-2}\right)^{-1}\mathbf X_{nk}\right),
\end{align*}
then we obtain\begin{small}
\begin{align}\label{al:3}
&{\rm E}_{k-1}\left|\left({\rm E}_k-{\rm E}_{k-1}\right)b_n\theta_k\sigma_k\right|^2\notag\\
\le&\frac{1}{y_p\left|z\right|v^2}{\rm E}_{k-1}\left|{\rm tr}\boldsymbol\varepsilon_k\right|^2\notag\\
\le&\frac{C}{y_p\left|z\right|v^2}\bigg[{\rm E}\left|{\rm tr}\left(\frac{1}{n}\boldsymbol{\phi}_k^{\prime}\boldsymbol{\bar{\phi}}_k-\mathbf I_2\right)\right|^2+\frac{\left|z\right|^2}{n^2}{\rm E}_{k-1}\left|{\rm tr}\left(\mathbf S_{n}-z\mathbf I_{2p}\right)^{-1}-{\rm Etr}\left(\mathbf S_{n}-z\mathbf I_{2p}\right)^{-1}\right|^2\notag\\
&+\frac{\left|z\right|^2}{n^2}{\rm E}_{k-1}\left|{\rm tr}\left(\mathbf S_{n}-z\mathbf I_{2p}\right)^{-1}-{\rm tr}\left(\mathbf S_{nk}-z\mathbf I_{2p-2}\right)^{-1}\right|^2+\frac{1}{n^2}\notag\\
&+\frac{1}{n^4}{\rm E}_{k-1}\left|{\rm tr}\boldsymbol{\phi}_k^{\prime}\mathbf X_{nk}^*\left(\mathbf S_{nk}-z\mathbf I_{2p-2}\right)^{-1}\mathbf X_{nk}\bar{\boldsymbol{\phi}}_k-{\rm tr}\mathbf X_{nk}^*\left(\mathbf S_{nk}-z\mathbf I_{2p-2}\right)^{-1}\mathbf X_{nk}\right|^2\bigg]\notag\\
\le&\frac{C}{y_p\left|z\right|v^2}\bigg(\frac{1}{n}+y_p^2\left|z\right|^2{\rm E}_{k-1}\left|s_p\left(z\right)-{\rm E}s_p\left(z\right)\right|^2+\frac{\left|z\right|^2}{n^2v^2}\notag\\
&+\frac{1}{n^4}{\rm E}_{k-1}{\rm tr}\left(\mathbf X_{nk}^*\left(\mathbf S_{nk}-z\mathbf I_{2p-2}\right)^{-1}\mathbf X_{nk}\mathbf X_{nk}^*\left(\mathbf S_{nk}-\overline z\mathbf I_{2p-2}\right)^{-1}\mathbf X_{nk}\right)\bigg)\notag\\
\le&\frac{C}{y_p\left|z\right|v^2}\left(\frac{\left|z\right|^2}{n}+\frac{\left|u\right|^2}{n^2 v}{\rm E}_{k-1}\Im{\rm tr}\left(\mathbf S_{nk}-z\mathbf I_{2p-2}\right)^{-1}+y_p^2\left|z\right|^2{\rm E}_{k-1}\left|s_p\left(z\right)-{\rm E}s_p\left(z\right)\right|^2\right)\notag\\
\le&\frac{C}{y_p\left|z\right|v^2}\left(\frac{\left|z\right|^2}{n}+\frac{y_p\left|u\right|^2}{n v}{\rm E}_{k-1}\Im s_p\left(z\right)+y_p^2\left|z\right|^2{\rm E}_{k-1}\left|s_p\left(z\right)-{\rm E}s_p\left(z\right)\right|^2\right)\notag\\
\le&\frac{C}{nv^3}{\rm E}_{k-1}\left(1+\Im s_p\left(z\right)\right)+\frac{C}{v^2}{\rm E}_{k-1}\left|s_p\left(z\right)-{\rm E}s_p\left(z\right)\right|^2.
\end{align}
\end{small}
Combining (\ref{al:4}), (\ref{al:5}), and (\ref{al:3}), for large n,  the first term on the
right-hand side  of (\ref{al:2}) is dominated by
\begin{align*}
C{n^{-2l}}{\rm E}\left(\sum_{k=1}^p{\rm E}_{k-1}\left|\gamma_k\right|^2\right)^l\le \frac{C}{n^{2l}v^{3l}}{\rm E}\left(1+\Im s_p\left(z\right)\right)^l+\frac{C}{n^lv^{2l}}{\rm E}\left|s_p\left(z\right)-{\rm E}s_p\left(z\right)\right|^{2l}.
\end{align*}
Similarly, by (\ref{al:4}), one gets
\begin{align}
&n^{-4l}\left|b_n\right|^{2l}{\rm E}\bigg|\bigg[{\rm tr}\boldsymbol{\phi}_k^{\prime}\mathbf X_{nk}^*\left(\mathbf S_{nk}-z\mathbf I_{2p-2}\right)^{-2}\mathbf X_{nk}\bar{\boldsymbol{\phi}}_k\notag\\
& \ \ \ \ \ \ \ \ \ \ \ \ -{\rm tr}\left(\mathbf X_{nk}^*\left(\mathbf S_{nk}-z\mathbf I_{2p-2}\right)^{-2}\mathbf X_{nk}\right)\bigg]\bigg|^{2l}\notag\\
\le&\frac{C}{n^{2l}}\left|b_n\right|^{2l}{\rm E}\bigg\{\varphi_{4l}{\rm tr}\left(\left(\mathbf S_{nk}-z\mathbf I_{2p-2}\right)^{-2}\mathbf S_{nk}\left(\mathbf S_{nk}-\overline z\mathbf I_{2p-2}\right)^{-2}\mathbf S_{nk}\right)^l\notag\\
&+\left[\varphi_4{\rm tr}\left(\left(\mathbf S_{nk}-z\mathbf I_{2p-2}\right)^{-2}\mathbf S_{nk}\left(\mathbf S_{nk}-\overline z\mathbf I_{2p-2}\right)^{-2}\mathbf S_{nk}\right)\right]^l\bigg\}\notag\\
\le&\frac{C}{n^{2l}y_p^l\left|z\right|^l}{\rm E}\bigg\{\varphi_{4l}v^{-2l}{\rm tr}\left(\mathbf I_{2p-2}+\left|z\right|^{2l}\left(\mathbf S_{nk}-z\mathbf I_{2p-2}\right)^{-l}\left(\mathbf S_{nk}-\overline z\mathbf I_{2p-2}\right)^{-l}\right)\notag\\
&+\varphi_4^lv^{-2l}\left[n+\left|z\right|^2{\rm tr}\left(\left(\mathbf S_{nk}-z\mathbf I_{2p-2}\right)^{-1}\left(\mathbf S_{nk}-\overline z\mathbf I_{2p-2}\right)^{-1}\right)\right]^l\bigg\}\notag\\
\le&\frac{C}{n^{2l}y_p^l\left|z\right|^l}{\rm E}\bigg\{n^lv^{-2l}\left(1+\left|z\right|^{2l}v^{-2l+1}\Im s_p\left(z\right)+\left|z\right|^{2l}n^{-1}v^{-2l}\right)\notag\\
&+\varphi_4^ln^lv^{-2l}\left[1+\left|z\right|^2v^{-1}\Im s_p\left(z\right)+\left|z\right|^{2}n^{-1}v^{-2}\right]^l\bigg\}\notag\\
\le&\frac{C}{n^{l}y_p^lv^{3l}}{\rm E}\left[1+v^{-l+1}\Im s_p\left(z\right)+\left(\Im s_p\left(z\right)\right)^l+n^{-1}v^{-l}\right]\notag\\
\le&\frac{C}{n^{l}y_p^lv^{3l}}{\rm E}\left[1+v^{-l+1}\Im s_p\left(z\right)+n^{-1}v^{-l}\right]
\end{align}
where the third inequality follows from the fact that $\varphi_{4l}\le Cn^{l-1}$.
And
\begin{align}
&{\rm E}\left|\left({\rm E}_k-{\rm E}_{k-1}\right)b_n\theta_k\sigma_k\right|^{2l}\notag\\
\le&\frac{C}{y_p^l\left|z\right|^lv^{2l}}{\rm E}\left|{\rm tr}\boldsymbol\varepsilon_k\right|^{2l}\notag\\
\le&\frac{C}{y_p^l\left|z\right|^lv^{2l}}\bigg[{\rm E}\left|{\rm tr}\left(\frac{1}{n}\boldsymbol{\phi}_k^{\prime}\boldsymbol{\bar{\phi}}_k-\mathbf I_2\right)\right|^{2l}+\frac{\left|z\right|^{2l}}{n^{2l}}{\rm E}\left|{\rm tr}\left(\mathbf S_{n}-z\mathbf I_{2p}\right)^{-1}-{\rm Etr}\left(\mathbf S_{n}-z\mathbf I_{2p}\right)^{-1}\right|^{2l}\notag\\
&+\frac{\left|z\right|^{2l}}{n^{2l}}{\rm E}\left|{\rm tr}\left(\mathbf S_{n}-z\mathbf I_{2p}\right)^{-1}-{\rm tr}\left(\mathbf S_{nk}-z\mathbf I_{2p-2}\right)^{-1}\right|^{2l}+\frac{1}{n^{2l}}\notag\\
&+\frac{1}{n^{4l}}{\rm E}\left|{\rm tr}\boldsymbol{\phi}_k^{\prime}\mathbf X_{nk}^*\left(\mathbf S_{nk}-z\mathbf I_{2p-2}\right)^{-1}\mathbf X_{nk}\bar{\boldsymbol{\phi}}_k-{\rm tr}\mathbf X_{nk}^*\left(\mathbf S_{nk}-z\mathbf I_{2p-2}\right)^{-1}\mathbf X_{nk}\right|^{2l}\bigg]\notag\\
\le&\frac{C}{y_p^l\left|z\right|^lv^{2l}}\bigg\{\frac{1}{n^l}+y_p^{2l}\left|z\right|^{2l}{\rm E}\left|s_p\left(z\right)-{\rm E}s_p\left(z\right)\right|^{2l}+\frac{\left|z\right|^{2l}}{n^{2l}v^{2l}}\notag\\
&+\frac{1}{n^{2l}}{\rm E}\bigg(\varphi_{4l}{\rm tr}\left(\left(\mathbf S_{nk}-z\mathbf I_{2p-2}\right)^{-1}\mathbf S_{nk}\left(\mathbf S_{nk}-\overline z\mathbf I_{2p-2}\right)^{-1}\mathbf S_{nk}\right)^l\notag\\
&+\left[\varphi_4{\rm tr}\left(\left(\mathbf S_{nk}-z\mathbf I_{2p-2}\right)^{-1}\mathbf S_{nk}\left(\mathbf S_{nk}-\overline z\mathbf I_{2p-2}\right)^{-1}\mathbf S_{nk}\right)\right]^l\bigg)\bigg\}\notag\\
\le&\frac{C}{y_p^l\left|z\right|^lv^{2l}}\bigg\{\frac{\left|z\right|^{2l}}{n^l}+\frac{1}{n^{2l}}{\rm E}\bigg(\varphi_{4l}{\rm tr}\left(\mathbf I_{2p-2}+\left|z\right|^{2l}\left(\mathbf S_{nk}-z\mathbf I_{2p-2}\right)^{-l}\left(\mathbf S_{nk}-\bar{z}\mathbf I_{2p-2}\right)^{-l}\right)\notag\\
&+\varphi_4^l\left[n+\left|z\right|^2{\rm tr}\left(\left(\mathbf S_{nk}-z\mathbf I_{2p-2}\right)^{-1}\left(\mathbf S_{nk}-\overline z\mathbf I_{2p-2}\right)^{-1}\right)\right]^l\bigg)\notag\\
&+y_p^{2l}\left|z\right|^{2l}{\rm E}\left|s_p\left(z\right)-{\rm E}s_p\left(z\right)\right|^{2l}\bigg\}\notag\\
\le&\frac{C}{n^ly_p^l\left|z\right|^lv^{2l}}\bigg(\left|z\right|^{2l}+v^{-2l+1}\left|z\right|^{2l}{\rm E}\Im s_p\left(z\right)+v^{-l}\left|z\right|^{2l}{\rm E}\left(\Im s_p\left(z\right)\right)^l
+\frac{1}{nv^{2l}}\bigg)\notag\\
&+\frac{C}{v^{2l}}{\rm E}\left|s_p\left(z\right)-{\rm E}s_p\left(z\right)\right|^{2l}\notag\\
\le&\frac{C}{n^lv^{3l}y_p^{l}}{\rm E}\left(1+v^{-l+1}\Im s_p\left(z\right)+\frac{1}{nv^l}\right)+\frac{C}{v^{2l}}{\rm E}\left|s_p\left(z\right)-{\rm E}s_p\left(z\right)\right|^{2l}.
\end{align}
Therefore, together with the two inequalities above, the second term on the right hand side of (\ref{al:2}) is bounded by
\begin{align*}
C{n^{-2l}}\sum_{k=1}^p{\rm E}\left|\gamma_k\right|^{2l}\le&\frac{C}{n^{3l-1}v^{3l}y_p^{l}}{\rm E}\left(1+v^{-l+1}\Im s_p\left(z\right)+\frac{1}{nv^l}\right)\\
&+\frac{C}{n^{2l-1}v^{2l}}{\rm E}\left|s_p\left(z\right)-{\rm E}s_p\left(z\right)\right|^{2l}.
\end{align*}
Consequently, we obtain
\begin{align}\label{al:6}
{\rm E}\left|s_p\left(z\right)-{\rm E}s_p\left(z\right)\right|^{2l}\le&\frac{C}{n^{2l}v^{3l}y_p^l}\left[1+{\rm E}\left(1+\Im s_p\left(z\right)\right)^l
+n^{-l+1}v^{-l+1}{\rm E}\Im s_p\left(z\right)\right]\notag\\
&+\frac{C}{n^lv^{2l}}{\rm E}\left|s_p\left(z\right)-{\rm E}s_p\left(z\right)\right|^{2l}\notag\\
\le&\frac{C}{n^{2l}v^{3l}y_p^l}\left[1+{\rm E}\left(\Im s_p\left(z\right)\right)^l
+n^{-l+1}v^{-l+1}{\rm E}\Im s_p\left(z\right)\right].
\end{align}
Now, we shall complete the proof of the lemma by using induction on $l$ and the inequality (\ref{al:6}).
\begin{description}
  \item[step 1] When $l=1$, by Lemma \ref{le:8.17} and Lemma \ref{B.22}, one has
\begin{align*}
{\rm E}\left|s_p\left(z\right)-{\rm E}s_p\left(z\right)\right|^{2}
\le&\frac{C}{n^{2}v^{3}y_p}\left[1+{\rm E}\left(\Im s_p\left(z\right)\right)\right]\\
\le&\frac{C}{n^{2}v^{3}y_p}\left[1+\left|{\rm E} s_p\left(z\right)-s_{y_p}\left(z\right)\right|+\left|s_{y_p}\left(z\right)\right|\right]\\
\le&\frac{C}{n^{2}v^{3}y_p}\left[1+\Delta/v+1/\left(\sqrt {y_p}v_{y_p}\right)\right]\\
\le&\frac{C}{n^{2}v^{4}y_p^2}\left(\Delta+v/v_{y_p}\right).
\end{align*}
  \item[step 2] In the final step, we need the case $l\in\left(\frac{1}{2},1\right)$. Therefore, we shall extend the lemma to $l\in\left(\frac{1}{2},1\right)$. By Lemma \ref{2.12} and the first step, it follows that
\begin{align*}
{\rm E}\left|s_p\left(z\right)-{\rm E}s_p\left(z\right)\right|^{2l}\le&\frac{C}{n^{2l}}{\rm E}\left(\sum_{k=1}^p\left|\gamma_k\right|^2\right)^l\\
\le&\frac{C}{n^{2l}}\left(\sum_{k=1}^p{\rm E}\left|\gamma_k\right|^2\right)^l\\
\le&{C}\left({\rm E}\left|s_p\left(z\right)-{\rm E}s_p\left(z\right)\right|^{2}\right)^l\\
\le&\frac{C}{n^{2l}v^{4l}y_p^{2l}}\left(\Delta+v/v_{y_p}\right)^l.
\end{align*}
  \item[step 3] Suppose that, for $l\in \left(2^{t-1},2^t\right],t=0,1,\cdots,k-1$, the lemma is true. Then consider the case $l\in\left(2^{k-1},2^k\right]$. By (\ref{al:6}), we have
\begin{align*}
{\rm E}\left|s_p\left(z\right)-{\rm E}s_p\left(z\right)\right|^{2l}
\le&\frac{C}{n^{2l}v^{3l}y_p^l}\bigg[1+{\rm E}\left|s_p\left(z\right)-{\rm E} s_p\left(z\right)\right|^l+\left|{\rm E} s_p\left(z\right)\right|^l\\
&
+n^{-l+1}v^{-l+1}\left|{\rm E} s_p\left(z\right)\right|\bigg]\\
\le&\frac{C}{n^{2l}v^{3l}y_p^l}\bigg[1+\frac{\left(\Delta+v/v_{y_p}\right)^{l/2}}{n^{l}v^{2l}y_p^{l}}+\left|\Delta/v+1/\left(\sqrt {y_p}v_{y_p}\right)\right|^l\\
&
+n^{-l+1}v^{-l+1}\left|\Delta/v+1/\left(\sqrt {y_p}v_{y_p}\right)\right|\bigg]\\
\le&\frac{C}{n^{2l}v^{4l}y_p^{2l}}\left(\Delta+v/v_{y_p}\right)^l.
\end{align*}
\end{description}
Then, the proof of the lemma is complete.
\end{proof}
\begin{lemma}\label{ex}
Under the conditions of Theorem \ref{th:1} and the additional assumption $\left\|x_{jk}\right\|\le n^{-1/4}$, for any fixed $t>0$, we have
\begin{align*}
\int_B^{\infty}\left|{\rm E}F_p\left(x\right)-F_{y_p}\left(x\right)\right|dx=o\left(n^{-t}\right),
\end{align*}
where $B=b_n+1=\sqrt{y_p}+2$.
\end{lemma}
\begin{proof}
In \cite{li2013extreme}, it has proved that, for any $\xi>0$ and $m=\left[\log n\right]$,
\begin{align*}
{\rm E}\left(\lambda_{\max}\left(\mathbf S_n\right)\right)^m\le\left(b+\xi\right)^m.
\end{align*}
Note that
\begin{align*}
1-F_p\left(x\right)\le I\left(\lambda_{\max}\left(\mathbf S_n\right)\ge x\right), \ for \ x\ge 0.
\end{align*}
Then, it follows that
\begin{align*}
&\int_B^{\infty}\left|{\rm E}F_p\left(x\right)-F_{y_p}\left(x\right)\right|dx
\le\int_B^{\infty}{\rm P}\left(\lambda_{\max}\left(\mathbf S_n\right)\ge x\right)\\
\le&\int_B^{\infty}\left(\frac{b+\xi}{x}\right)^mdx=O\left(\left(\frac{b+\xi}{B}\right)^{m-1}\right)\\
=&o\left(n^{-t}\right)
\end{align*}
which completes the proof of this lemma.
\end{proof}

\section{Appendix}
In this section, to be self-contained, we shall present some existing results which will be used in the proof of the main theorems.
\begin{lemma}[Theorem A.44 in\cite{bai2010spectral}]\label{le:1}
Let $\mathbf A$ and $\mathbf B$ be two $m\times k$ complex matrices. Then,
\begin{align*}
\left\|F^{\mathbf{AA^*}}-F^{\mathbf{BB^*}}\right\|\le\frac{1}{m}{\rm rank}\left(\mathbf A-\mathbf B\right)
\end{align*}
where $\left\|g\right\|=\sup_x\left|g\left(x\right)\right|$.
\end{lemma}
\begin{lemma}[Bernstein's inequality]\label{le:2}
If $\mathbf Y_1,\cdots,\mathbf Y_k$ are independent random variables with mean zeros and uniformly bounded by $K$, then, for any $\varepsilon>0$,
\begin{align*}
{\rm P}\left(\left|\sum_j\mathbf Y_j\right|\ge\varepsilon\right)\le2\exp\left\{-\varepsilon^2/\left[2\left(B_n^2+K\varepsilon\right)\right]\right\}
\end{align*}
where $B_n^2={\rm Var}\left(\sum_j\mathbf Y_j\right)$.
\end{lemma}
\begin{lemma}[Theorem A.45 in \cite{bai2010spectral}]\label{le:3}
Let $\mathbf A$ and $\mathbf B$ be two $m\times m$ Hermitian matrices. Then,
\begin{align*}
L\left(F^{\mathbf A},F^{\mathbf B}\right)\le\left\|\mathbf A-\mathbf B\right\|_2
\end{align*}
where $L$ is the Levy distance between two two-dimensional distribution functions $F$ and $G$ defined by
\begin{align*}
L\left(F,G\right)=\inf\left\{\varepsilon:F\left(\xi-\varepsilon,\eta-\varepsilon\right)-\varepsilon\le G\left(\xi,\eta\right)\le F\left(\xi+\varepsilon,\eta+\varepsilon\right)+\varepsilon\right\}.
\end{align*}
\end{lemma}
\begin{lemma}[Theorem A.47 in \cite{bai2010spectral}]\label{le:4}
Let $\mathbf A$ and $\mathbf B$ be two $m\times k$ complex matrices. Then,
\begin{align*}
L\left(F^{\mathbf{AA^*}},F^{\mathbf{BB^*}}\right)\le2\left\|\mathbf A\right\|_2\left\|\mathbf {A-B}\right\|_2+\left\|\mathbf {A-B}\right\|_2^2.
\end{align*}
\end{lemma}
\begin{lemma}[Lemma B.19 in \cite{bai2010spectral}]\label{le:5}
Let $F_1$, $F_2$ be distribution functions and let $G$ satisfy $\sup_{x}\left|G\left(x+\theta\right)-G\left(x\right)\right|\le g\left(\theta\right)$, for all $\theta$, where $g$ is an increasing and continuous function such that $g\left(0\right)=0$. Then
\begin{align*}
\left\|F_1-G\right\|_2\le3\max\left\{\left\|F_2-G\right\|_2,L\left(F_1,F_2\right),g\left(L\left(F_1,F_2\right)\right)\right\}.
\end{align*}
\end{lemma}
\begin{remark}[Lemma 8.14 in \cite{bai2010spectral}]\label{re:1}
For the M-P law with index $y\le 1$, the function g can be taken as $g\left(v\right)=2v/\left(y\left(\sqrt a+\sqrt v\right)\right)$.
\end{remark}
\begin{lemma}[Inversion formula for block matrix ]\label{le:6}
Suppose that the matrix $\mathbf \Sigma$ is nonsingular and has the partition as given by $\left(\begin{array}{cc}\Sigma_{11}&\Sigma_{12}\\ \Sigma_{21}&\Sigma_{22}\end{array}\right)$. If $\Sigma_{11}$ is also singular, then, the inverse of $\mathbf \Sigma$ has the from
\begin{align*}
\mathbf \Sigma^{-1}=\left(\begin{array}{cc}
\Sigma_{11}^{-1}+\Sigma_{11}^{-1}\Sigma_{12}\Sigma_{22.1}^{-1}\Sigma_{21}\Sigma_{11}^{-1}&-\Sigma_{11}^{-1}\Sigma_{12}\Sigma_{22.1}^{-1}\\
-\Sigma_{22.1}^{-1}\Sigma_{21}\Sigma_{11}^{-1}&\Sigma_{22.1}^{-1}\end{array}\right)
\end{align*}
where $\Sigma_{22.1}^{-1}=\Sigma_{22}-\Sigma_{21}\Sigma_{11}^{-1}\Sigma_{12}$.
\end{lemma}

\begin{lemma}[Lemma 2.18 in \cite{yin2013rates}]\label{B.26}
Let ${\mathbf A}=\left(a_{jk}\right)_{j,k=1}^{2n}$ be a $2n\times2n$ non-random matrix and ${\mathbf X}=(x_1',\cdots,x_n')'$ be a random quaternion vector of independent entries.
Assume that ${\rm E}x_j=0$, ${\rm E}\left\|x_j\right\|^2=1$, and ${\rm E}\left\|x_j\right\|^l\leq \varphi_l$. Then, for any $m\geq 1$, we have
$${\rm E}\left|{\rm tr}{\mathbf X}^*{\mathbf A}{\mathbf X}-{\rm tr} {\mathbf A} \right|^m \leq C_m \left(\left(\varphi_4{\rm tr}\left({\mathbf A}{\mathbf A^*}\right)\right)^{m/2}+\varphi_{2m}{\rm tr}\left({\mathbf A}{\mathbf A^*}\right)^{m/2}\right),$$
where $C_m$ is a constant depending on $m$ only.
\end{lemma}
\begin{lemma}[see $\left(A.1.12\right)$ in \cite{bai2010spectral}]\label{lemma:7}
Let $z = u + iv, v > 0, $ and let $A$ be an $n \times n$ Hermitian matrix.  ${A_k}$ be the k-th major sub-matrix of $A$ of order $(n-1)$, to be the matrix resulting from the $k$-th row and column from $A$.  Then
$$\left| {{\rm tr}{{\left(A - z{I_n}\right)}^{ - 1}} -{ \rm tr}{{\left({A_k} - z{I_{n - 1}}\right)}^{ - 1}}} \right| \le \frac{1}{\upsilon}.$$
\end{lemma}

\begin{lemma}[Lemma 2.12 in \cite{bai2010spectral}]\label{2.12}
Let $\left\{{\tau_k}\right\}$ be a complex martingale difference sequence with respect to the increasing $\sigma$-fields $\mathcal{F}_k$. Then, for $p > 1, $ ${\rm E}{\left| {\sum {{\tau_k}} } \right|^p} \le {K_p}{\rm E}{({\sum {\left| {{\tau_k}} \right|} ^2})^{p/2}}.$
\end{lemma}

\begin{lemma}[Rosenthal's inequality ]\label{lemma:5}
Let $ X_i $ are independent with zero means,  then we have, for some constant $C_k$:
$${\rm E}\left|\sum X_i\right|^{2k} \leq C_k\left(\sum{\rm E}\left|X_i\right|^{2k}+\left(\sum {\rm E}\left|X_i\right|^2\right)^k\right). $$
\end{lemma}

\begin{lemma}[Lemma 8.17 in \cite{{bai2010spectral}}]\label{le:8.17}
For the Stieltjes transform of the M-P law, we have
\begin{align*}
\left|s_{y_p}\left(z\right)\right|\le\frac{\sqrt 2}{\sqrt {y_p}v_{y_p}}.
\end{align*}
\end{lemma}

\begin{lemma}[Lemma B.22 in \cite{bai2010spectral}]\label{B.22}
Let $G$ be a function of bounded variation. Let $g(z)$ denote its Stieltjes transform. When $z = u + iv,$ with  $v>0,$ we have
$$\sup_u \left| g(z)\right|\leq \pi v^{-1}\left\|G\right\|.$$
\end{lemma}

\begin{lemma}[Lemma 8.17 in \cite{bai2010spectral}]\label{leb}
For all $z\in\mathbb{C^+}$, when $\left|\delta\right|\le v/\left[v_{y_p}10\left(A+1\right)^2\right]$, we have
\begin{align*}
\left|b_n\right|\le\frac{2}{\sqrt{y_p\left|z\right|}}.
\end{align*}
\end{lemma}

\begin{lemma}[Lemma 8.21 in \cite{bai2010spectral}]\label{8.21}
If $\left|\delta_n\right|<v/\left[v_{y_p}10\left(A+1\right)^2\right]$ for all $\left|z\right|<A$, then there is a constant $C$ such that
\begin{align*}
\Delta\le Cv/v_{y_p}
\end{align*}
where $A$ is defined in Lemma \ref{bai1993} for the M-P law with index $y\le 1$.
\end{lemma}

\begin{lemma}\label{ex1}
For $v>n^{-1/2}$, we have
\begin{align*}
\sup_x\int_{\left|u\right|<v}\left|F_{y_p}\left(x+u\right)-F_{y_p}\left(x\right)\right|du
\le\frac{11\sqrt{2\left(1+y\right)}}{3\pi y}v^2/v_{y_p},
\end{align*}
where $F_{y_p}$ is the M-P law with index $y\le 1$.
\end{lemma}


\end{document}